\documentclass[preprint,12pt]{elsarticle}
\usepackage{amssymb}
\usepackage{amsmath}

\usepackage{ogonek}

\newcommand{\ba}{\begin{array}}
\newcommand{\ea}{\end{array}}
\newcommand{\bi}{\begin{itemize}}
\newcommand{\ei}{\end{itemize}}
\newcommand{\bc}{\begin{center}}
\newcommand{\ec}{\end{center}}
\newcommand{\bfr}{\begin{flushright}}
\newcommand{\efr}{\end{flushright}}

\newtheorem{theorem}{Theorem}[section]
\newtheorem{lemma}{Lemma}[section]
\newtheorem{corollary}{Corollary}[section]

\journal{}

\begin{document}
\begin{frontmatter}
  \title{On a convexity problem}
  \author{Bogdan Gavrea}
  \address{Department of Mathematics, Technical University of Cluj-Napoca,
          Str. Memorandumului nr. 28, 400114, Cluj-Napoca, Romania}
  \begin{abstract}
     This work is a continuation of what was done in \cite{BGavrea2018} and it is 
   strongly connected to the work done in \cite{AbelRasa2018}.
  \end{abstract}
\end{frontmatter}

\section{Introduction}\label{sec:intro}
Let $n\in \mathbb{N}$ and let $\Pi_n$ denote the set of all polynomials of degree $\le n $. The fundamental 
Bernstein polynomials of degree $n$ are given by: 
$$
    b_{n,k}(x) = {n \choose k } x^k (1-x)^{n-k}, \; k=0,1,...,n. 
$$
In (\cite{Ref1}, Problem $2$, pp. $164$), I. Ra\c{s}a (\cite{Ref1}, Problem $2$, pp. $164$), came up with the following problem: \emph{Prove or disprove the following inequality:}
\begin{equation} \label{eq:1_1}
  \sum_{i=0}^n \sum_{j=0}^n \left[ b_{n,i}(x) b_{n,j}(x)+ b_{n,i}(y) b_{n,j}(y)- 2  b_{n,i}(x) b_{n,j}(y)
		\right] f\left(\frac{i+j}{2n}\right) \ge 0,
\end{equation}
for any convex function $f\in C[0,1]$ and any $x,y \in [0,1]$. 
In \cite{Ref7}, by using a probabilistic approach, J. Mrowiec, T. Rajba  and S. W\k{a}sowicz,  
 gave a positive answer to the above 
problem and proved the following generalization of inequality (\ref{eq:1_1}).
\begin{theorem}[\cite{Ref7}, Theorem 12]\label{th:Miro}
   Let $m,n\in \mathbb{N}$ with $m\ge 2$. Then, 
   \begin{eqnarray}
	\sum_{i_1,...,i_m=0}^n 
	\left[ b_{n,i_1}(x_1)...b_{n,i_m}(x_1)+...+ b_{n,i_1}(x_m)...b_{n,i_m}(x_m)
	\right. \nonumber\\
	\left.  -m b_{n,i_1}(x_1)...b_{n,i_m}(x_m) 
	\right]
	f\left( \frac{i_1+...+i_m}{mn} \right) \ge 0,
	\label{eq:1_2}
   \end{eqnarray}
for any convex function $f\in C[0,1]$ and any $x_1,...,x_m \in [0,1]$. 
\end{theorem}
An elementary proof of (\ref{eq:1_1}), was given recently by Abel
in \cite{Ref2}, where it is  shown that a type (\ref{eq:1_1})  inequality holds also 
for the Mirakyan-Favard-Sz\'{a}sz (\cite{Ref2}, Theorem 5)  and for the Baskakov operators (\cite{Ref2}, 
Theorem 6). 

In \cite{BGavrea2018}, we proved a  type (\ref{eq:1_1})  inequality for a large class of operators defined
in the following way. Let $I$ be one of the intervals $[0,\infty)$ or $[0,1]$. Let $g_n:I\times D \to \mathbb{C}$, 
$ D= \{ z\in  \mathbb{C} \;|\; |z|\le R\}$, $R> 1$
be a function with the property that for any fixed $x\in I$, the function $g_n(x,\cdot)$ is an analytic function 
on $D$,  
\begin{eqnarray}\label{eq:gn}
g_n(x,z) & =  &\sum_{k=0}^\infty a_{n,k}(x) z^k\nonumber\\
a_{n,k}(x)& \ge &  0, \forall k\ge 0 \label{eq:gn}\\
g_n(x,1) & = & 1, \forall x \in I. \nonumber
\end{eqnarray}
In what follows, let $I=[0,\infty)$. The case $I=[0,1]$ follows in the same way.
Let $\mathcal{F}$ be a linear set of functions defined on the interval $I$ and let
$\{A_t\}_{t\in I}$ 
be a set of real linear 
positive functionals defined on $\mathcal{F}$ with the property that for any $f\in \mathcal{F}$, the series 
\begin{equation}\label{eq:2_3}
    L_{n,A}(f)(x) := \sum_{k=0}^\infty a_{n,k}(x) A_{\frac{k}{n}}(f).
\end{equation}
is convergent for any $x\in I$. The identity (\ref{eq:2_3}) defines a \emph{positive linear operator}. The function $g_n$ will be referred to as the \emph{generating function} for the operator $L_{n,A}$ relative to 
the set of functionals $\{A_t\}_{t\in I}$.

In what follows , we assume that the linear positive functionals $\{A_t\}_{t\in I}$ are such that 
$L_{n,A}$ is well defined for any $f\in \mathcal{F}$ and any $x\in I$, the set of all real polynomials 
$\Pi\subseteq \mathcal{F}$ and every functional $A_t$ has the following properties:
\begin{enumerate}
   \item[i)] $A_t(e_0) = 1, t\in I$
   \item[ii)] $A_t(e_1) = at +b,  t\in I$, where $a$ and $b$ are two real numbers independent of $t$
	       and $e_i(x) = x^i$, $x\in I$, $i\in \mathbb{N}$.
\end{enumerate}
In \cite{BGavrea2018}, we obtained the following result: {\it if 
\begin{equation}\label{eq:ineq_divdiff}
		\left[ \frac{k}{n}, \frac{k+1}{n}, \frac{k+2}{n}; A_t(f)\right] \ge 0
	\end{equation}
	and 
	\begin{equation}\label{eq:ineq_derivs}
		\frac{d^k}{dz^k} \left. \left[\frac{g_n(x,z)-g_n(y,z)}{z-1} \right]^2\right|_{z=0} \ge 0,
	\end{equation}
	for any $k\in \mathbb{N}$ 
	and all $x,y\in I$, 
	 then 
	$ A(f) \ge 0$. 
}
Here, for $x,y \in  I $ fixed, the functional $A$ is defined by
\begin{equation*}
	A(f)= \sum_{i=0}^\infty \sum_{j=0}^\infty 
	\left[ a_{n,i}(x) a_{n,j}(x) + a_{n,i}(y) a_{n,j}(y) - 2 a_{n,i}(x) a_{n,j}(y)\right]
	A_{\frac{i+j}{2n}}(f), 
\end{equation*}
The following result (\cite{BGavrea2018}, Corollary 3.2) is useful to verify inequality (\ref{eq:ineq_derivs}), 

\noindent {\it Let $x, y\in I$ be two distinct numbers. Assume that conditions i) and ii) above hold, 
\begin{equation}\label{eq:4_25}
	   \frac{g_n(x,z)- g_n(y,z)}{z-1} = \sum_{k=0}^\infty \beta_{n,k}(x,y) z^k
	\end{equation}
	and $\mbox{sgn}  \;\beta_{n,k}(x,y)$  is the same for all $k\in \mathbb{N}$,
	then (\ref{eq:ineq_derivs}) is satisfied.
}

For  $m\in \mathbb{N}$, $m\ge 2$ and $x\in I^m$, $x= (x_1,...,x_m)$, we define  the functionals:
\begin{eqnarray}
	C_m(f)& = & \sum_{i_1,...,i_m =0}^\infty \left[ a_{n,i_1}(x_1)...a_{n,i_m}(x_1) 
	+...+a_{n,i_1}(x_m)...a_{n,i_m}(x_m)\right. \nonumber\\
		& - & \left.  m a_{n,i_1}(x_1)...a_{n,i_m}(x_m)\right] A_{\frac{i_1+...+i_m}{mn}} (f)
	\label{eq:Cmf}
\end{eqnarray}
In \cite{BGavrea2018}, Theorem 4.1, we have proved the following result:

\noindent{If (\ref{eq:ineq_divdiff}) and (\ref{eq:ineq_derivs}) hold, then 
$$
C_m(f) \ge 0
$$
for any $m\in \mathbb{N}$, $m\ge 2$. }

Applications, such as Bernstein type operators, Mirakyan-Favard-Sz\'{a}sz type operators,  Meyer-K\"{o}nig and Zeller type operators, were considered in \cite{BGavrea2018}. 

Let us assume that that the generating functions $g_n$, $n\in \mathbb{N}^*$ are of the form 
\begin{equation}\label{eq:gnphi}
   g_n(x,z) = \phi^n(x,z),
\end{equation}
where $\phi:I \times D \to \mathbb{C}$ is such that $\phi(x,\cdot)$ is an analytic function and the function $g_n$ given by (\ref{eq:gnphi}) satisfies conditions (\ref{eq:gn}). Under these assumptions, we have 
\begin{equation}\label{eq:anmk}
   \sum_{i_1+...+i_m =k} a_{n,i_1}(x)... a_{n,i_m}(x) = a_{nm,k}(x).
\end{equation}
The above identity implies that 
\begin{equation}\label{eq:CmLmn}
  C_m(f) = \sum_{k=1}^m L_{mn,A}(f) (x_k) - 
	m \sum_{i1,...,i_m=0}^\infty
	a_{i_1}(x_1)...a_{i_m}(x_m) f\left(\ \frac{i_1+...+i_m}{mn}\right)
\end{equation}
Let us assume that the sequence $\left(L_{n,A}\right)_{n\in \mathbb{N}^*}$ preserves convexity. More precisely, 
we assume that for every convex function  $f\in \mathcal{F}$, $L_{n,A}(f)$, $n\in \mathbb{N}^*$ is convex too. Under this assumption, we have 
\begin{equation}\label{eq:ineqLnm}
L_{nm,A}(f)\left(\frac{x_1+...+x_m}{m}\right) \le \sum_{k=1}^m \frac{L_{nm,A}(f)(x_k)}{m}.
\end{equation}
For the Bernstein operators,  in \cite{AbelRasa2018}, the following  problem was studied:

\noindent
{\it Prove that
  \begin{equation}\label{eq:BernsteinEven}
     B_{2n}(f)\left(\frac{x+y}{2}\right) \ge 
     \sum_{i=0}^n \sum_{j=0}^n b_{n,i}(x) b_{n,j}(x) f\left(\frac{i+j}{2n}\right),
  \end{equation}
for all convex $f\in C[0,1]$ and $x,y\in [0,1]$.}

\noindent
A probabilistic solution was found by A. Komisarski and T. Rajba, 
\cite{KomiRajba}. In \cite{AbelRasa2018}, U. Abel and I. Ra\c{s}a gave an analytic proof to the following theorem. 
\begin{theorem}[\cite{AbelRasa2018}, Theorem 1] Let $n,m\in \mathbb{N}$. If $f\in C[0,1]$ is a convex function, then the
inequality 
$$
   B_{mn}(f)\left(\frac{1}{m} \sum_{\nu=1}^m x_\nu \right)
  \ge \sum_{i_1= 0}^n ...\sum_{i_m = 0}^ n
	\left( \prod_{\nu =1}^m b_{n,i_\nu}(x_\nu)\right) 
	f\left( \frac{1}{mn} \sum_{\nu=1}^m i_\nu\right)
$$
is valid for all $x_1,...,x_m\in [0,1]$.
\end{theorem}
The purpose of this paper is to give sufficient conditions for the generating functions $g_n$, $n \in \mathbb{N}$, such 
that the functional $\mathbb{B}_m:\mathcal{F}\to \mathbb{R}$, 
\begin{equation}\label{eq:Bmf}
   \mathbb{B}_{m}(f) = L_{mn,A}(f)\left(\frac{x_1+...+x_m}{m}\right)
		- \sum_{i_1,...,i_m=0}^\infty a_{i_1}(x_1)... a_{i_m}(x_m) A_{\frac{i_1+...+i_m}{mn}}(f)
\end{equation}
is nonnegative for any function $f\in \mathcal{F}$ for which
\begin{equation}\label{eq:divdiffAt}
   \left[ \frac{k}{n}, \frac{k+1}{n}, \frac{k+2}{n}; A_t (f) \right] \ge 0
\end{equation}
and for any $x = (x_1,...,x_m)\in I^m$ and any $k\in \mathbb{N}$. It is immediate to see, from (\ref{eq:CmLmn}), that if $\mathbb{B}_m(f)\ge 0$, then $C_m(f)\ge 0$ as well. 

\section{Main results}\label{sec:Main}

\begin{theorem}\label{th:main1}
   Let $f\in \mathcal{F}$ be such that inequality (\ref{eq:divdiffAt}) holds. If 
   \begin{equation}\label{eq:dkdzk}
	\frac{d^k}{dz^k} \left. \left[\frac{g_{nm}\left(\frac{x_1+...+x_m}{m},z\right)-
         g_n(x_1,z)...g_n(x_m,z)}{z-1} \right]\right|_{z=0} \ge 0
   \end{equation}
   for any $k\in \mathbb{N}$ and any $x=(x_1,...,x_m)\in I^m$,then 
   $$
	\mathbb{B}_m(f)\ge 0.
   $$
   If the reverse of (\ref{eq:dkdzk}) holds for any $k\in \mathbb{N}$ and any $x=(x_1,...,x_m)\in I^m$, then 
    $$
	\mathbb{B}_m(f)\le 0.
   $$
\end{theorem}
\begin{proof}
	We note that 
	$$
		\mathbb{B}_m(e_0) = \mathbb{B}_m(e_1) = 0. 
	$$
	On the other hand, we have
	$$
		\mathbb{B}_m(f) = L_{mn,A}(f)\left(\frac{x_1+....+x_m}{m}\right)-
		\sum_{k=0}^\infty \alpha_{n,k}(x) A_{\frac{k}{mn}}(f),
	$$
	where
	$$
		\alpha_{n,k}(x) = \sum_{i_1+...+i_m=k} a_{n,i_1}(x_1)...a_{n,i_m}(x_m).
	$$
         So
	 $$
		\mathbb{B}_m(f) =  \sum_{k=0}^\infty
		\left[ a_{mn,k}\left(\ \frac{x_1+...+x_m}{m} \right)
			-  \sum_{i_1+...+i_m=k} a_{n,i_1}(x_1)...a_{n,i_m}(x_m)\right]
		f\left(\frac{k}{m+n}\right).
	$$
	We note that
	\begin{eqnarray}
		g_{mn}\left(\ \frac{x_1+...+x_m}{m},z \right) - g_n(x_1,z)...g_n(x_m,z) = \nonumber\\
		\sum_{k=0}^\infty \left[ a_{mn,k}\left(\ \frac{x_1+...+x_m}{m} \right)
			-  \sum_{i_1+...+i_m=k} a_{n,i_1}(x_1)...a_{n,i_m}(x_m)\right] z^k.\label{eq:gmngn}
	\end{eqnarray}
	From (\ref{eq:gmngn}), we get
	\begin{eqnarray}
		a_{mn,k}\left( \frac{x_1+...+x_m}{m} \right)
			-  \sum_{i_1+...+i_m=k} a_{n,i_1}(x_1)...a_{n,i_m}(x_m) = \nonumber\\
			\frac{1}{2\pi}  \int_0^{2\pi} 
			\left[ g_{mn}\left(\frac{x_1+...+x_m}{m}, e^{i \theta}\right) - 
			g_n\left(x_1, e^{i\theta}\right)...g_n\left(x_m, e^{i\theta}\right)\right] e^{-ik\theta} d\theta,
			\label{eq:amn}
	\end{eqnarray}
	for any $k\in \mathbb{N}$. From (\ref{eq:amn}), by using the same technique as 
	in the proof of Theorem 4.1, \cite{BGavrea2018}, we get 
	\begin{equation}\label{eq:BmBernstein}
		\mathbb{B}_m(f) = \frac{2}{nm} \sum_{k=2}^\infty \mathbb{B}_m\left(\left|\cdot - \frac{k-1}{nm} \right|\right)
		\left[ \frac{k-2}{mn}, \frac{k-1}{mn},  \frac{k}{mn}; A_t(f) \right],
	\end{equation}
	where 
	\begin{equation}\label{eq:BmAbsValue}
   		\mathbb{B}_m\left(\left| \cdot - \frac{k-1}{mn}\right|\right) = \frac{1}{nm} \frac{1}{(k-2)!} 
	\frac{d^{k-2}}{dz^{k-2}} \left.
	 \frac{E^2_m(x,z)}{(z-1)^2} \right|_{z=0}
	\end{equation}
	and 
	\begin{equation}\label{eq:Emn}
		E_{m}(x,z) = g_{mn}\left(\frac{x_1+...+x_m}{m}, z\right) - g_n(x_1,z)...g_n(x_m,z).
	\end{equation}
	 Equations (\ref{eq:BmBernstein}),  (\ref{eq:BmAbsValue}) and (\ref{eq:4_25}) conclude our proof. 
\end{proof}
In what follows we are interested in whether there exists a large  class of linear positive operators for which 
$A(f) \ge 0$, whenever (\ref{eq:ineq_divdiff}) and (\ref{eq:ineq_derivs}) are satisfied and $\mathbb{B}_m(f)\ge 0$ 
or $\mathbb{B}_m(f)\le 0$.
\section*{Mastroianni type operators}
We denote by $C_2([0,\infty))$ the function space  
$$C_2([0,\infty)) := \left\{ f\in C([0,\infty)) \; : \; \exists  \lim_{x\to\infty} \frac{f(x)}{1+x^2} < \infty\right\}.$$
Let $\left( \varphi_n\right)_{n\in \mathbb{N}}$ be a sequence of real functions defined on $[0,\infty)$, 
$\varphi_n \in C^\infty[0,\infty)$, $n \in \mathbb{N}$ that are strictly monotone and satisfy the following conditions:
	\begin{eqnarray}
		\varphi_n(0) = 1, n \in \mathbb{N} \label{eq:24}\\
		(-1)^n \varphi_n^{(k)} (x) \ge 0, n \in \mathbb{N}^*, k \in \mathbb{N}, x\ge 0 \label{eq:25}\\
		\forall (n,k) \in \mathbb{N}\times \mathbb{N}, \exists \; p(n,k) \in \mathbb{N} \mbox{ and }\label{eq:26}\\
		\exists \alpha_{n,k}:[0,\infty) \to \mathbb{R} \mbox{ such that } \forall x\ge 0, \forall  i \in \mathbb{N}^*
		 \nonumber\\
		\varphi_{n}^{(i+k)} (x)  = (-1)^k \varphi_{p(n,k)}^{(i)} (x) \alpha_{n,k}(x) \mbox{ and }
		\lim_{n\to \infty} \frac{n}{p(n,k)} = \lim_{n\to \infty} \frac{\alpha_{n,k}(x)}{n^k} = 1\nonumber
	\end{eqnarray}
G. Mastroianni, in \cite{Mastroianni}, introduced for any  $n\in \mathbb{N}^*$, the operators
 $M_n:C_2([0,\infty))\to C([0,\infty))$, defined by 
$$
	M_n(f)(x) = \sum_{k=0}^\infty \frac{(-1)^k}{k!} x^k \varphi_n^{(k)}(x) f \left(\frac{k}{n}\right).
$$
Let $(A_t)_{t\in I}$ be a set of linear positive functionals defined on the linear set of functions 
$\mathcal{F}$, satisfying conditions i) and ii) above and such that for every $f\in \mathcal{F}$, the series 
	\begin{equation}\label{eq:27}
		M_{n,A}(f) (x) := \sum_{k=0}^\infty (-1)^k \frac{x^k \varphi_n^{(k)}(x)}{k !} A_{\frac{k}{n}}(f)
	\end{equation}
converges. We will assume that $\Pi_2\subseteq \mathcal{F}$. 

\noindent {\bf Remark.} If $\mathcal{F}= C_2([0,\infty))$, then $M_{n,A}(f)$ is well defined.
\begin{lemma}\label{lem:2_2}
	If for any $x\in[0,\infty)$, the function $g_n(x,\cdot)=\varphi(x(1-\cdot)($ is analytic
	 in  $D= \{z\in \mathbb{C} :\; |z|<R\}$,
	 $R>1$, then $g_n$ is a generating function for $M_{n,A}$.
\end{lemma}
\begin{proof}
 	We have 
	$$
		\frac{d^k}{dz^k} g_n(x,z) = (-1)^k x^k \varphi_n^{(k)}(x(1-z)) 
	$$
	and therefore 
	$$
		g_n(x,z) = \sum_{k=0}^\infty (-1)^k x^k \frac{\varphi_n^{(k)}(x)}{k!} z^k.
	$$
\end{proof}
\begin{theorem}\label{th:2_3}
	Let $x,y\in [0,\infty)$, $ x\ne y$. If 
	$$
		\frac{g_n(x,z)-g_n(y,z)}{z-1} = \sum_{k=0}^\infty \beta_{n,k}(x,y) z^k,
	$$
	then $sgn \beta_{n,k}(x,y)$ is the same for all $k\in \mathbb{N}$.
\end{theorem}
\begin{proof}
	We have 
	$$
		\frac{g_n(x,z)-g_n(y,z)}{z-1} = - \sum_{p=0}^\infty 
		\frac{(-1)^p x^p \varphi_n^{(p)}(x) - (-1)^p y^p \varphi_n^{(p)}(y)}{p!} 
		\sum_{m=0}^\infty z^m.
	$$
	It follows that 
	\begin{equation}\label{eq:28}
		\beta_{n,k}(x,y) = - \sum_{p=0}^k 
		\frac{(-1)^p x^p \varphi_n^{(p)}(x) - (-1)^p y^p \varphi_n^{(p)}(y)}{p!} .
	\end{equation}
	Let us consider the function $h_{n,k}:[0,\infty) \to \mathbb{R}$ defined by 
	$$
		h_{n,k}(t) = - \sum_{p=0}^k \frac{(-1)^p t^p \varphi_n^{(p)}(t)}{p!}.
	$$
	We have 
	\begin{eqnarray*}
		h'_{n,k}(t) & =  & - \sum_{p=0}^k \frac{(-1)^p p t^{p-1} \varphi_n^{(p)}(t)}{p!}
					  - \sum_{p=0}^k \frac{(-1)^p t^p \varphi_n^{(p+1)}(t)}{p!}\\
				& = & \sum_{p=0}^{k-1} \frac{(-1)^p t^p \varphi_n^{(p+1)}(t)}{p!}
					- - \sum_{p=0}^k \frac{(-1)^p t^p \varphi_n^{(p+1)}(t)}{p!}\\
				& = & \frac{(-1)^{p+1} t^p \varphi_n^{(p+1)}(t)}{p!} \ge 0, \forall t \in [0,\infty),
				 \forall p\in \mathbb{N}.
	\end{eqnarray*}
	But 
	$$
		\beta_{n,k}(x,y) = h_{n,k}(x) - h_{n,k}(y)
	$$
 	and therefore 
	$$
		sgn \beta_{n,k}(x,y) = sgn(x-y), \forall x,y \in [0,\infty),
	$$
	which concludes our proof. 
\end{proof}
\begin{corollary}\label{cor:2_4}
  	Let $M_{n,A}$ be a family of Mastroianni type operators and let $f\in \mathcal{F}$. 
	If 
	$$
		\left[
			\frac{k}{n}, \frac{k+1}{n}, \frac{k+2}{n}; A_t(f)
		\right] \ge 0, \forall k \in \mathbb{N}, 
	$$
	then for all the functionals $C_m$, given by (\ref{eq:Cmf}), 
	with 
	$$
		a_{n,i_k}(x_i) = \frac{(-1)^{i_k} x_i^{i_k} \varphi_n^{(i_k)}(x_i)}{i_k!},
	$$
	we have $C_m(f) \ge 0$.
\end{corollary}
\section*{ Examples}
\begin{enumerate}
	\item[1.]{\bf  Bernstein type operators} are Mastroianni type operators with the functions 
	$\left(\varphi_n\right)_{n\in\mathbb{N}}$ defined by $\varphi_n(x) = (1-x)^n$ and the generating functions
	$g_n(x,t)$ given by 
	$$
		g_n(x,t) = (1-x +tx)^n.
	$$
	\item[2.] {\bf Mirakyan-Favard-Sz\'{a}sz type operators},  $S_{n,A}$, 
	$$
		S_{n,A}(f)(x) = e^{-nx} \sum_{k=0}^\infty \frac{(nx)^k}{k!} A_{\frac{k}{n}}(f)
	$$
	are obtained for 
	$
		\varphi_n(x) = e^{-nx}, x\ge 0 \mbox{ and } 
		g_n(x,z) = e^{-nx(1-z)}.
	$
	\item[3.]{\bf Baskakov type operators}, $V_{n,A}$, 
	$$
		V_{n,A}(f)(x) = (1+x)^{-n} \sum_{k=0}^\infty {n+k-1 \choose k} \left(\frac{x}{1+x}\right)^k 
		A_{\frac{k}{n}}(f) 
	$$
	are obtained for $\varphi_n(x) = (1+x)^{-n}$, $n \in \mathbb{N}^*$ and $g_n(x,z) = (1+x-xz)^{-n}$.
	\item[4.] {\bf Sz\'{a}sz-Schurer type operators}, $s_{n,p,A}$, 
	$$
		S_{n,p,A}(f)(x) = e^{-(n+p)x} \sum_{k=0}^\infty \frac{ (n+p)^k }{k!} A_{\frac{k}{n}}(f)
	$$
	are obtained for $\varphi_n(x) = e^{-(n+p)x}$ and $g_n(x,z) = e^{-(n+p)x(1-z)}$.
\end{enumerate}
We note that in the above examples the generating functions are of the following form:
$$
	g_n(x,z) = g_1^n(x,z),
$$
where $g_1(x,z)$ is the generating function for the operator $L_{1,A}$. In these case $E_m(x,z)$ given by 
(\ref{eq:Emn}) can be written in the following form 
$$
	E_m(x,z) = g_1^{nm} \left(\frac{x_1+...+x_m}{m},z\right) - \left( g_1(x_1,z)...g_1(x_m,z)\right)^n 
$$
and by using Theorem \ref{th:main1}, we get the following result
\begin{theorem}\label{th:main2}
	Let $f\in \mathcal{F}$ be a function with the property that 
	$$
		\left[\frac{k}{n}, \frac{k+1}{n}, \frac{k+2}{n}; A_t(f)\right]\ge 0, \forall k\in \mathbb{N}.
	$$
	If 
	\begin{equation}\label{eq:29}
		\frac{d^k}{dz^k} \left. \left[\frac{g_{1}^m\left(\frac{x_1+...+x_m}{m},z\right)-
        		 g_1(x_1,z)...g_1(x_m,z)}{z-1} \right]\right|_{z=0} \ge 0
  	 \end{equation}
	 for all $k\in \mathbb{N}$ and all $x=(x_1,...,x_m)\in I ^m$, 
	then $\mathbb{B}_m(f)\ge 0$. If  the reverse of  inequality (\ref{eq:29}) is satisfied
	for all $k\in \mathbb{N}$ and all $x=(x_1,...,x_m)\in I ^m$, then 
	$\mathbb{B}_m(f)\le 0$.
\end{theorem}
\section*{Concluding remarks}
 We mention below a few consequences of Theorem \ref{th:main2}.
\begin{enumerate}
	\item The Bernstein type operators verify (\ref{eq:29}). In this case $g_n(x,z) = 1-x +zx$ and 
	inequality (\ref{eq:29}) follows from Gusi\'c, \cite{Gusic}, Theorem 1 (see also \cite{Tarartkova}, Equation (2)), 
	where the following representation is given
	\begin{equation}\label{eq:31}
		\left(\sum_{\nu= 1}^m a_\nu\right)^m - m^m \sum_{\nu=1}^m a_\nu
		= \sum_{1\le i <j \le m} (a_i-a_j)^2 P_{i,j}(a_1,...,a_m).
	\end{equation}
	In (\ref{eq:31}), $P_{i,j}$ are some homogeneous polynomials of degree $n-2$ with non-negative coefficients.
	Identity (\ref{eq:31}) was used by Abel and Ra\c{s}a in \cite{AbelRasa2018} for the classical Bernstein operators. 
	\item For $g_1(x,z) = e^{-x(1-z)}$, we get 
	$$
		\mathbb{B}_m(f) = C_m(f), \; m\in \mathbb{N}^*. 
	$$
	\item  In the case of Baskakov type operators, we have
	$$
		g_1(x,z) = \frac{1}{1+x-xz}.
	$$
	Using now (\ref{eq:31}), it follows that the reverse of inequality (\ref{eq:29}) is satisfied. Therefore, if $f \in 
	\mathcal{F}$ and 
	$$
		\left[
			\frac{k}{n}, \frac{k+1}{n}, \frac{k+2}{n}; A_t(f)
		\right] \ge 0, \forall k \in \mathbb{N},
	$$
	then the Baskakov type operatotrs satisfy the  following inequalities
	\begin{equation*}
		V_{n,A}(f) \left(\frac{x_1+...+x_m}{m}\right) \le \sum_{i_1=0}^\infty...\sum_{i_m=0}^\infty
		\prod_{\nu=1}^m a_{n,i_\nu}(x_\nu) A_{\sum_{\nu=1}^m i_\nu/m }
	\end{equation*}
	and 
	\begin{eqnarray*}
		\sum_{i_1,...,i_m=0}^\infty \left[ a_{n,i_1}(x_1)..a_{n,i_m}(x_1).+....+a_{n,i_1}(x_1m..a_{n,i_m}(x_m)\right]
		A_{\frac{i_1+...+i_m}{nm}}(f)\\
		\ge m \sum_{i_1,...,i_m=0}^\infty a_{n,i_1}(x_1)....a_{n,i_m}(x_m)A_{\frac{i_1+...+i_m}{nm}}(f)
	\end{eqnarray*}
\end{enumerate}


\end{document}